\documentclass[nointlimits,11pt,oneside]{amsart}
\usepackage{amssymb,cases,enumitem, verbatim}
\usepackage{xcolor}
\usepackage{hyperref}
\usepackage[T1]{fontenc}
\usepackage[utf8]{inputenc}

\hypersetup{
	colorlinks=true,
	linkcolor=blue,
	citecolor=blue,
	filecolor=green,
	urlcolor=cyan,
	bookmarks=true
}

\makeatletter
\renewcommand*{\eqref}[1]{%
	\hyperref[{#1}]{\textup{\tagform@{\ref*{#1}}}}%
}
\makeatother

\setlist[enumerate,1]{label={\textup{(\roman*)}}}

\usepackage[%
	a4paper,
	total={16cm,23cm},
	left=3cm, top=3cm,
	marginparsep=2pt]
{geometry}

\theoremstyle{plain}
\newtheorem{theorem}{Theorem}[section]
\newtheorem{lemma}[theorem]{Lemma}
\newtheorem{corollary}[theorem]{Corollary}
\newtheorem{proposition}[theorem]{Proposition}
\theoremstyle{definition}
\newtheorem{remark}[theorem]{Remark}

\newtheorem{definition}[theorem]{Definition}

\numberwithin{equation}{section}

\begin{document}

\title[On the sums of r.i.~quasi-Banach function spaces and amalgams]{On the sums of rearrangement-invariant~quasi-Banach function spaces and their relationship to amalgams}
\author{Dalimil Pe{\v s}a}

\address{Dalimil Pe{\v s}a, Technische Universität Chemnitz, Faculty of Mathematics, 09107 Chemnitz, Germany
	and
	Department of Mathematical Analysis, Faculty of Mathematics and Physics, Charles University, Sokolovsk\'a~83, 186~75 Praha~8, Czech Republic}
\email{dalimil.pesa@mathematik.tu-chemnitz.de}
\urladdr{0000-0001-6638-0913}

\subjclass[2020]{46E30}
\keywords{sums of spaces, rearrangement-invariant quasi-Banach function spaces, representation theorem, Wiener--Luxemburg amalgam spaces}

\thanks{This research was supported by the grant 23-04720S of the Czech Science Foundation.}

\begin{abstract}
	In this paper we consider the properties of sums of rearrangement-invariant quasi-Banach function spaces, with the focus being on rearrangement-invariance and the Fatou property. In our first main result, we show that the quasinorm of the sum is in many cases equivalent to a rearrangement-invariant quasinorm by providing a weaker version of the Luxemburg-type representation. In our second main result, we show that the sum can be in some cases characterised as a Wiener--Luxemburg amalgam of the two constituent spaces, thus providing a sufficient condition for the sum being a rearrangement-invariant quasi-Banach function space.

\end{abstract}

\date{\today}

\maketitle

\makeatletter
   \providecommand\@dotsep{2}
\makeatother

\section{Introduction} \label{SectIntro}
The purpose of this paper is to examine how the sums of rearrangement-invariant quasi-Banach function spaces inherit the defining properties of this class from their parent spaces. For the sake of completeness we also provide the corresponding results for the intersection; however, those results are much simpler and thus less interesting. After introducing the setting and gathering the known or easily obtained results, we shall observe that the interesting properties are the rearrangement-invariance and the Fatou property, as neither of the two transfers easily to the sum (at least in the general case); this is the content of Section~\ref{SectIntro} at the end of which we discuss how we address those questions. The remaining two sections then contains our main results. Namely, Section~\ref{SecRI&Repre} addresses the question of rearrangement-invariance as well as that of the Luxemburg-type representation, while Section~\ref{SecAmalgams} concerns itself with the question of the Fatou property, where our result is a consequence of characterising the sums of certain pairs of spaces as the appropriate Wiener--Luxemburg amalgams of said spaces.

\subsection{The setting}
Let $\mathcal{M}(\mathcal{R}, \mu)$ denote the set of all (equivalence classes of) measurable functions on the resonant measure space $(\mathcal{R}, \mu)$ (for the definition of resonant measure spaces see e.g.~\cite[Chapter~2, Definition~2.3]{BennettSharpley88}; see also \cite[Chapter~2, Theorem~2.7]{BennettSharpley88} for a crucial characterisation of the property). A \emph{quasi-Banach function norm} is then a functional $\lVert \cdot \rVert \colon \mathcal{M}(\mathcal{R}, \mu) \rightarrow [0, \infty]$ that satisfies $\lVert \, \lvert f \rvert \, \rVert = \lVert f \rVert$ for all $f \in \mathcal{M}(\mathcal{R}, \mu)$ and its restriction to $\mathcal{M}_+(\mathcal{R}, \mu)$, the cone of non-negative measurable functions, satisfies the following axioms:
\begin{enumerate}[label=\textup{(Q\arabic*)}]
	\item \label{Q1} it is a quasinorm, in the sense that it satisfies the following three conditions:
	\begin{enumerate}[ref=(\theenumii)]
		\item \label{Q1a} it is absolutely homogeneous, i.e.~$\forall a \in \mathbb{C} \; \forall f \in \mathcal{M}_+ : \lVert af \rVert = \lvert a \rvert \lVert f \rVert$,
		\item \label{Q1b} it satisfies  $\lVert f \rVert = 0 \Leftrightarrow f = 0$ $\mu$-a.e.,
		\item \label{Q1c} there is a constant $C\geq 1$, called the modulus of concavity of $\lVert \cdot \rVert$, such that it is subadditive up to this constant, i.e.
		\begin{equation*}
			\forall f,g \in \mathcal{M}_+ : \lVert f+g \rVert \leq C(\lVert f \rVert + \lVert g \rVert).
		\end{equation*}
	\end{enumerate}
\end{enumerate}
\begin{enumerate}[label=\textup{(P\arabic*)}, series=P]
	\setcounter{enumi}{1}
	\item \label{P2} it has the \emph{lattice property}, i.e.~if some $f, g \in \mathcal{M}_+$ satisfy $f \leq g$ $\mu$-a.e., then also $\lVert f \rVert \leq \lVert g \rVert$,
	\item \label{P3} it has the \emph{Fatou property}, i.e.~if  some $f_n, f \in \mathcal{M}_+$ satisfy $f_n \uparrow f$ $\mu$-a.e., then also $\lVert f_n \rVert \uparrow \lVert f \rVert $,
	\item \label{P4} $\lVert \chi_E \rVert < \infty$ for all $E \subseteq \mathcal{R}$ satisfying $\mu(E) < \infty$,
\end{enumerate} 
A \emph{Banach function norm} is a quasi-Banach function norm that further satisfies the following two conditions:
\begin{enumerate}[label=\textup{(P\arabic*)}]
	\setcounter{enumi}{0}
	\item \label{P1} it is a norm, in the sense that it satisfies the following three conditions:
	\begin{enumerate}[ref=(\theenumii)]
		\item \label{P1a} it is absolutely homogeneous, i.e.~$\forall a \in \mathbb{C} \; \forall f \in \mathcal{M}_+ : \lVert a f \rVert = \lvert a \rvert \lVert f \rVert$,
		\item \label{P1b} it satisfies $\lVert f \rVert = 0 \Leftrightarrow f = 0$  $\mu$-a.e.,
		\item \label{P1c} it is subadditive, i.e.~$\forall f,g \in \mathcal{M}_+  :  \lVert f+g \rVert \leq \lVert f \rVert + \lVert g \rVert$,
	\end{enumerate}
	\setcounter{enumi}{4}
	\item \label{P5} for every $E \subseteq \mathcal{R}$ satisfying $\mu(E) < \infty$ there exists some finite constant $C_E$, dependent only on $E$, such that the inequality $ \int_E f \: d\mu \leq C_E \lVert f \rVert $ is true for all $f \in \mathcal{M}_+$.
\end{enumerate}
A (quasi-)Banach function norm is called \emph{rearrangement-invariant}, abbreviated r.i., if it also satisfies
\begin{enumerate}[label=\textup{(r.i.)}]
	\item \label{r.i.} $\lVert f \rVert = \lVert g \rVert$ for all $f,g \in \mathcal{M}(\mathcal{R}, \mu)$ such that $f^* = g^*$,
\end{enumerate}
where $f^*$ is the non-increasing rearrangement of $f \in \mathcal{M}(\mathcal{R}, \mu) $ (see e.g.~\cite[Chapter~2, Section~1]{BennettSharpley88}). Finally, when we have some (r.i.)~(quasi-)Banach function norm $\lVert \cdot \rVert_X$, we define the corresponding \emph{(r.i.)~(quasi-)Banach function space} $X$ as
\begin{equation*}
	X = \{f \in \mathcal{M}(\mathcal{R}, \mu); \; \lVert f \rVert_X < \infty \}.
\end{equation*}

The setting of r.i.~Banach function spaces is very classical (for an exhaustive treatment, we recommend the books \cite{BennettSharpley88}, \cite{KreinPetunin82}, and \cite{Zaanen67}). However, it fails to cover many natural function spaces, such as $L^{1, \infty}$ or $L^p$, $p<1$, to name the most obvious examples. Hence, in recent years, the more general framework of r.i.~quasi-Banach function spaces has attracted a growing interest (for overviews of the recent development and applications we recommend the papers \cite{LoristNieraeth23}, \cite{MusilovaNekvinda24}, \cite{NekvindaPesa24}, and the references therein). Let us just point out one crucial result that is very closely related to the problem at hand, i.e.~the recent extension of the Luxemburg representation theorem obtained in \cite[Theorem~3.1]{MusilovaNekvinda24} and slightly modified \cite[Proposition~3.2]{PesaRepreACqN} (see also \cite{Luxemburg67} or \cite[Chapter~2, Theorem~4.10]{BennettSharpley88} for the classical version). This theorem poses, that for every r.i.~quasi-Banach function norm $\lVert \cdot \rVert_X$ over $(\mathcal{R}, \mu)$ there is an r.i.~quasi-Banach function norm $\lVert \cdot \rVert_{\overline{X}}$ over $([0,\infty), \lambda)$ (where $\lambda$ is the classical $1$-dimensional Lebesgue measure) that satisfies for every $f \in \mathcal{M}(\mathcal{R}, \mu)$
\begin{equation} \label{EqRepre}
	\lVert f \rVert_X = \lVert f^* \rVert_{\overline{X}}.
\end{equation}
This functional is then called a \emph{representation quasinorm} and the corresponding r.i.~quasi-Banach function space $\overline{X}$ is a \emph{representation space}. It is worth noting that the representation space is not necessarily unique, whence we will from now on always use the notation $\overline{X}$ for the specific representation space constructed in \cite[Definition~3.1]{PesaRepreACqN} (which is itself heavily based on the construction performed in \cite[Proof of Theorem~3.1]{MusilovaNekvinda24}) and we will call this space \emph{the canonical representation space} of $X$. Finally, it has been observed in both \cite[Section~3]{MusilovaNekvinda24} and \cite[Section~3]{PesaRepreACqN} that when the space $X$ is an r.i.~Banach function space, then the canonical representation space coincides with the representation space obtained from the classical Luxemburg construction introduced in \cite{Luxemburg67} (see also \cite[Chapter~2, Theorem~4.10]{BennettSharpley88} for a modern presentation).

\subsection{The questions}
From their inception, r.i.~spaces have always been closely tied with interpolation. They indeed provide a natural setting, as every quasi-Banach function space over $(\mathcal{R}, \mu)$ is continuously embedded into $\mathcal{M}_0(\mathcal{R}, \mu)$, the set of (equivalence classes of) functions that are finite $\mu$-a.e.~equipped with the topology of convergence in measure on sets of finite measure (see \cite[Theorem~3.4]{NekvindaPesa24}). Thence it naturally follows that every pair of quasi-Banach function spaces over the same measure space forms a compatible couple (see e.g.~\cite[Definition~1.1]{BennettSharpley88}, \cite[Section~2.3]{BerghLofstrom76}, or \cite[Chapter~1, Definition~3.1]{KreinPetunin82}). In the light of this observation, we may provide a somewhat simpler definition of the sum and intersection of spaces:
\begin{definition}
	Let $A$ and $B$ be (r.i.)~(quasi-)Banach function spaces over the same resonant measure space $(\mathcal{R}, \mu)$. Then
	\begin{equation*}
		A \cap B = \{f \in \mathcal{M}(\mathcal{R}, \mu); \; f \in A, f \in B\}
	\end{equation*}
	equipped with the functional 
	\begin{equation*}
		\lVert \cdot \rVert_{A \cap B} = \max \{\lVert \cdot \rVert_A, \, \lVert \cdot \rVert_B\},
	\end{equation*}
	while
	\begin{equation*}
		 A + B = \{f \in \mathcal{M}(\mathcal{R}, \mu); \; \exists f_A \in A, f_B \in B: f = f_A + f_B\}
	\end{equation*}
	equipped with the functional (defined for every $f \in \mathcal{M}(\mathcal{R}, \mu)$)
	\begin{equation*}
		\lVert f \rVert_{A + B} = \inf_{\stackrel{f_A \in A, \, f_B \in B}{f = f_A + f_B}} \lVert f_A \rVert_A + \lVert f_B \rVert_B.
	\end{equation*}
\end{definition}

The question then naturally presents itself whether these two spaces inherits the properties of r.i.~quasi-Banach function spaces. As it turns out, at least some parts of the answer are rather clear:

\begin{lemma} \label{LemmIntSumBasic}
	Let $A, B$ be r.i.~quasi-Banach function spaces over a resonant measure space. Then:
	\begin{enumerate}
		\item \label{LemmIntSumBasic_i} $A \cap B$ is an r.i.~quasi-Banach function space. Furthermore, if at least one of the quasinorms $\lVert \cdot \rVert_A$ and $\lVert \cdot \rVert_B$ satisfies the property \ref{P5}, then so does $\lVert \cdot \rVert_{A \cap B}$.
		\item \label{LemmIntSumBasic_ii} $A + B$ is a quasi-Banach space. Moreover, the quasinorm $\lVert \cdot \rVert_{A+B}$ satisfies for every $f \in \mathcal{M}$ that $\lVert f \rVert_{A+B} = \lVert \, \lvert f \rvert \, \rVert_{A+B}$ and it also satisfies the axioms \ref{P2} and \ref{P4} of (quasi-)Banach function norms. The axiom \ref{P5} is satisfied if and only if both $\lVert \cdot \rVert_A$ and $\lVert \cdot \rVert_B$ satisfy it.
	\end{enumerate}
\end{lemma}

\begin{proof}
	The validity of \ref{LemmIntSumBasic_i} is easily verified directly from the definition.
	
	As for \ref{LemmIntSumBasic_ii}, that $A + B$ is a quasi-Banach space is known, see e.g.~\cite[Section~3.10]{BerghLofstrom76} or \cite{PeetreSparr72}. The arguments that the quasinorm $\lVert \cdot \rVert_{A+B}$ satisfies $\lVert f \rVert_{A+B} = \lVert \, \lvert f \rvert \, \rVert_{A+B}$ for every $f \in \mathcal{M}$, as well as the axioms \ref{P2} and \ref{P4} of (quasi-)Banach function norms, are both easy and virtually identical to those for the normed case. That it has the property \ref{P5} whenever both $\lVert \cdot \rVert_A$ and $\lVert \cdot \rVert_B$ do so is immediate, the ``only if'' part follows from \cite[Theorem~3.12]{NekvindaPesa24}.
\end{proof}

Furthermore, for the more classical setting of r.i.~Banach function spaces, the existing literature also provides the rest of the solution (although we have been unable to find a full and explicit formulation of the statement). 

\begin{theorem} \label{ThmSumIntRIBFS}
	Let $A, B$ be r.i.~Banach function spaces. Then both $A \cap B$ and $A + B$ are r.i.~Banach function spaces. Moreover,
	\begin{align*}
		(A + B)' &= A' \cap B', \\
		(A \cap B)' &= A' + B',
	\end{align*}
	with equal norms, where $X'$ is the associate space of $X$.
\end{theorem}

For the definition of associate spaces (also called Köthe duals) and the related theory in the classical setting of Banach function spaces we recommend \cite[Chapter~1]{BennettSharpley88}. There is also an extension of this theory to the more general setting of quasi-Banach function spaces, but this extension is not relevant to our topic.

\begin{proof}
	That $A \cap B$ is an r.i.~Banach function space follows immediately from the definition. It is also quite easy to verify that $(A + B)' = A' \cap B'$ with equal norms.
	
	On the other hand, it follows from \cite[Lemma~1.12]{CwikelNilsson03} (see also \cite{Lozanovskii78}) that $(A \cap B)' = A' + B'$ with equal norms. Hence, the classical Lorentz--Luxemburg theorem (see e.g.~\cite[Theorem~3]{Luxemburg55} or \cite[Chapter~1, Theorem~2.7]{BennettSharpley88}) implies $A+B = A'' + B'' = (A' \cap B')'$ with equal norms, from which it follows that $A+B$ is a Banach function space (see e.g.~\cite[Theorem~3.1]{GogatishviliSoudsky14}). Furthermore, since the r.i.~Banach function spaces $A, B$ are exact interpolation spaces between $L^1$ and $L^{\infty}$ (see \cite[Chapter~3, Theorem~2.12]{BennettSharpley88} for the precise formulation we use, or \cite[Theorem~3]{Calderon66} for the original result), it is clear that so is $A+B$. Thus, it follows from \cite[Chapter~3, Theorem~2.12]{BennettSharpley88} that $A+B$ is r.i.
\end{proof}

However, for the general case, the Fatou property \ref{P3} and rearrangement-invariance \ref{r.i.} are more complicated.

As for \ref{r.i.}, interpolation is no longer usable, as the spaces are generally not even intermediate spaces between $L^1$ and $L^{\infty}$, much less interpolation ones. Moreover, the fact that a given space may fail to include $L^1 \cap L^{\infty}$ means that the more direct method applied e.g.~in \cite[Chapter~2, §4.1]{KreinPetunin82} cannot be used either (even if we were to restrict ourselves to the case $([0,\infty), \lambda)$). We tackle this problem in Theorem~\ref{ThmSumRIqBFS}, where we show that in the case when the underlying measure space is non-atomic, the space $\overline{A} + \overline{B}$ serves as an ``almost representation'' of $A+B$, meaning that we have a representation formula analogous to \eqref{EqRepre} but with an equivalence (up to multiplicative constants) instead of equality. While this does not prove that the sum itself is rearrangement-invariant, it shows that there is an equivalent rearrangement-invariant quasinorm with nice properties. As for the remaining case of resonant measure spaces, i.e.~the completely atomic spaces where all atoms have the same measure (see e.g.~\cite[Chepter~2, Theorem~2.7]{BennettSharpley88}), we were only able to prove the analogous statement (i.e.~Theorem~\ref{ThmSumRIqBFS_atom}) under the additional assumption that both $A$ and $B$ satisfy the Hardy--Littlewood--Pólya principle. We explain this in more details below in Section~\ref{SecRI&Repre}.

Relatedly, even in the cases when it is known that the sum or intersection is rearrangement-invariant, the Luxemburg-type representation of said spaces seems to be untreated by the existing literature. We fill this gap in Proposition~\ref{PropIntRIqBFS} and Theorem~\ref{ThmSumRIBFS} where we show, respectively, that $\overline{A} \cap \overline{B}$ is the canonical representation space of $A \cap B$ and that $\overline{A} + \overline{B}$ is a representation space of $A + B$, in both cases with equal (quasi)norms and for all cases of resonant measure spaces.

As for the Fatou property \ref{P3}, the above presented approach is indirect: it shows that the norm is equal to the norm in a certain associate space, i.e.~(for $f \in \mathcal{M}(\mathcal{R}, \mu)$)
\begin{equation} \label{EqAssoc}
	\lVert f \rVert_{A+B} = \lVert f \rVert_{(A' \cap B')'} = \sup_{\substack{g \in \mathcal{M}(\mathcal{R}, \mu) \\ \lVert g \rVert_{A' \cap B'} \leq 1}} \int_{\mathcal{R}} \lvert fg \rvert \: d\mu,
\end{equation}
whence it is clear that the Fatou property \ref{P3} holds. However, it is also clear that the functional on the right-hand side of \eqref{EqAssoc} satisfies the triangle inequality with constant one (regardless of the set over which the supremum is taken). Hence, this method is unusable for the non-normable spaces. Our contribution is contained in Section~\ref{SecAmalgams}, where we provide a partial solution via the Wiener--Luxemburg amalgams. Similarly to the classical case, we do not provide a direct proof but instead show that the quasinorm $\lVert \cdot \rVert_{A+B}$ is equivalent to a certain r.i.~quasi-Banach function norm, i.e.~a functional about which we know a~priori that it has the Fatou property \ref{P3}. However, this approach requires additional assumptions on certain forms of ``embeddings'' between $A$ and $B$, as well as their rearrangement-invariance. We also assume that the underlying measure space is non-atomic and of infinite measure; however this assumption is quite natural in this context as otherwise the above mentioned ``embeddings'' between $A$ and $B$ would render the problem trivial.

\section{Rearrangement-invariance and representation}\label{SecRI&Repre}

The main result of this Section is the following representation-type theorem for the sum of r.i.~quasi-Banach function spaces over non-atomic measure spaces.

\begin{theorem} \label{ThmSumRIqBFS}
	Let $A, B$ be r.i.~quasi-Banach function spaces over a non-atomic measure space $(\mathcal{R}, \mu)$. Denote by $\lVert \cdot \rVert_{\overline{A}}$ and $\lVert \cdot \rVert_{\overline{B}}$ the corresponding canonical representation quasinorms. Then we have for every $f \in \mathcal{M}(\mathcal{R}, \mu)$ that
	\begin{equation}  \label{ThmSumRIqBFS:E0}
		\lVert f \rVert_{A+B} \approx \lVert f^* \rVert_{\overline{A} + \overline{B}}.
	\end{equation}
	Furthermore, the functional $f \mapsto \lVert f^* \rVert_{\overline{A} + \overline{B}}$ always has the properties \ref{Q1}, \ref{P2}, \ref{P4}, and \ref{r.i.}, while the validity of \ref{P5} is equivalent to the same property for $A+B$.
\end{theorem}

We recall that the relation ``$\approx$'' means that the ratio of the left-hand side and the right-hand side is bounded from both above and below by some positive and finite constants that depend on the quasinorms in question but not on the functions being measured. We will also use the symbol ``$\lesssim$'' to denote the corresponding one-directional estimate. Further, the proof requires us to work with the dilation operator $D_{\frac{1}{2}}$ which maps a given function $f(\cdot) \in \mathcal{M}([0,\infty), \lambda)$ to the function $f\left( \frac{\cdot}{2} \right)$, see e.g.~\cite[Section~3.4]{NekvindaPesa24} for details.

\begin{proof}
	Recalling Lemma~\ref{LemmIntSumBasic}, we may, without loss of generality, consider only non-negative functions $f \in \mathcal{M}(\mathcal{R}, \mu)$.
	
	Let $f^* \in \overline{A} + \overline{B}$. Then $f$ is finite $\mu$-a.e. (see~\cite[Theorem~3.4]{NekvindaPesa24}), whence
	\begin{equation*}
		\lim_{t \to \infty} f^*(t) = \alpha \in [0, \infty).
	\end{equation*}
	Put
	\begin{align*}
		E_{f,\alpha} &= \left\{ x \in \mathcal{R}; \; f(x) > \alpha \right\},
	\end{align*}
	where we consider an arbitrary representative of $f$, and also
	\begin{align*}
		f_{\alpha} &= \max\{ f - \alpha, \, 0 \}, \\
		f_{\infty} &= \min \{f, \, \alpha\}.
	\end{align*}	
	Then we clearly have that $f = f_{\alpha} + f_{\infty}$, $f_{\infty} \leq \alpha$ $\mu$-a.e., $f_{\alpha} = f_{\alpha} \chi_{E_{f,\alpha}}$ $\mu$-a.e., and
	\begin{align*}
		f_{\alpha}^* &= \max\{f^* - \alpha, \, 0\} = f^* - \alpha, &\lim_{t \to \infty} f_{\alpha}^*(t) = 0.
	\end{align*}
	
	It now follows from the Ryff's theorem (see e.g.~\cite[Chapter~2, Corollary~7.6]{BennettSharpley88} or \cite{Ryff70}) that there is a measure preserving mapping $\sigma_{f_{\alpha}}$ from the support of $f_{\alpha}$ (i.e.~$E_{f,\alpha}$) onto the support of $f_{\alpha}^*$ such that $f_{\alpha} = f_{\alpha}^* \circ \sigma_{f_{\alpha}}$ on $E_{f,\alpha}$. Let us now consider an arbitrary decomposition of $f_{\alpha}^*$ into functions in $\overline{A}, \overline{B}$, i.e.
	\begin{align} \label{ThmSumRIqBFS:E1}
		f_{\alpha}^* &= (f_{\alpha}^*)_{\overline{A}} + (f_{\alpha}^*)_{\overline{B}}, &(f_{\alpha}^*)_{\overline{A}} \in \overline{A} \text{ and } (f_{\alpha}^*)_{\overline{B}} \in \overline{B}.
	\end{align}
	We may assume without loss of generality that $(f_{\alpha}^*)_{\overline{A}}$ and $(f_{\alpha}^*)_{\overline{B}}$ are both zero $\lambda$-a.e.~outside the support of $f_{\alpha}^*$. Then clearly the functions
	\begin{align*}
		f_{\alpha,A} &= \begin{cases}
			(f_{\alpha}^*)_{\overline{A}} \circ \sigma_{f_{\alpha}} &\text{on } E_{f,\alpha}, \\
			0 &\text{otherwise;}
		\end{cases} \\
		f_{\alpha,B} &= \begin{cases}
			(f_{\alpha}^*)_{\overline{B}} \circ \sigma_{f_{\alpha}} &\text{on } E_{f,\alpha}, \\
			0 &\text{otherwise.}
		\end{cases}
	\end{align*}
	satisfy $f_{\alpha,A} \in A$, $f_{\alpha,B} \in B$, $f_{\alpha} = f_{\alpha,A} + f_{\alpha,B}$, and thus
	\begin{equation*} 
		\lVert f_{\alpha} \rVert_{A+B} \leq \lVert f_{\alpha,A} \rVert_A + \lVert f_{\alpha,B} \rVert_B = \lVert (f_{\alpha,A})^* \rVert_{\overline{A}} + \lVert (f_{\alpha,B})^* \rVert_{\overline{B}} = \lVert (f_{\alpha}^*)_{\overline{A}} \rVert_{\overline{A}} + \lVert (f_{\alpha}^*)_{\overline{B}} \rVert_{\overline{B}}.
	\end{equation*}
	In the last two steps, we use, respectively, \eqref{EqRepre} and the fact that $\sigma_{f_{\alpha}}$ is measure-preserving. By taking an infimum over the decompositions of the form \eqref{ThmSumRIqBFS:E1}, we obtain
	\begin{equation} \label{ThmSumRIqBFS:E2}
		\lVert f_{\alpha} \rVert_{A+B} \leq \lVert f_{\alpha}^* \rVert_{\overline{A}+\overline{B}}.
	\end{equation}
	
	As for $f_{\infty}$, we may assume that $\alpha >0$, as otherwise there is nothing to prove. Then clearly $\mu(\mathcal{R}) = \infty$. Since $(g_A + g_B)^* \leq D_{\frac{1}{2}}g_A^* + D_{\frac{1}{2}}g_B^*$ for every $g_A \in A$ and $g_B \in B$ (see e.g.~\cite[Chapter~2, Proposition~1.7]{BennettSharpley88}), we observe that at least one of the spaces $A, B$ must contain a function $g$ satisfying
	\begin{equation*}
		\lim_{t \to \infty} g^*(t) > 0.
	\end{equation*}
	The situation is symmetrical, so let us assume that this space is $A$. Then \cite[Corollary~3.5]{MusilovaNekvinda24} implies that $\chi_{\mathcal{R}} \in A$. We thus compute
	\begin{equation} \label{ThmSumRIqBFS:E3}
		\lVert f_{\infty} \rVert_{A+B} \leq \alpha \lVert \chi_{\mathcal{R}} \rVert_{A} \leq \frac{\lVert \chi_{\mathcal{R}} \rVert_{A}}{\lVert \chi_{[0, \infty)} \rVert_{\overline{A}+\overline{B}}} \lVert \alpha \chi_{[0, \infty)} \rVert_{\overline{A}+\overline{B}} \leq \frac{\lVert \chi_{\mathcal{R}} \rVert_{A}}{\lVert \chi_{[0, \infty)} \rVert_{\overline{A}+\overline{B}}} \lVert f^* \rVert_{\overline{A}+\overline{B}}.
	\end{equation}	
	Here, we use the property \ref{P2} of $\overline{A}+\overline{B}$ as established in Lemma~\ref{LemmIntSumBasic} and the observation that $\chi_{[0, \infty)} = \chi_{\mathcal{R}}^* \in \overline{A} \hookrightarrow \overline{A}+\overline{B}$.
	
	Putting \eqref{ThmSumRIqBFS:E2} and \eqref{ThmSumRIqBFS:E3} together, we obtain
	\begin{equation*}
		\lVert f \rVert_{A+B} \lesssim \lVert f_{\alpha} \rVert_{A+B} + \lVert f_{\infty} \rVert_{A+B} \lesssim \lVert f_{\alpha}^* \rVert_{\overline{A}+\overline{B}} + \lVert f^* \rVert_{\overline{A}+\overline{B}} \lesssim \lVert f^* \rVert_{\overline{A}+\overline{B}}.
	\end{equation*}
	
	On the other hand, let $f \in A + B$ and fix some $f_A \in A$, $f_B \in B$, such that $f = f_A + f_B$. Then $f^* \leq D_{\frac{1}{2}}f_A^* + D_{\frac{1}{2}}f_B^*$ (see e.g.~\cite[Chapter~2, Proposition~1.7]{BennettSharpley88}). Now, the dilation operator is bounded on the representation spaces $\overline{A}, \overline{B}$ (see \cite[Theorem~3.23]{NekvindaPesa24}), whence it follows that $D_{\frac{1}{2}}f_A^* \in \overline{A}$ and $D_{\frac{1}{2}}f_B^* \in \overline{B}$ and that
	\begin{equation*}
		\lVert f^* \rVert_{\overline{A}+\overline{B}} \leq \lVert D_{\frac{1}{2}}f_A^* + D_{\frac{1}{2}}f_B^* \rVert_{\overline{A}+\overline{B}} \leq \lVert D_{\frac{1}{2}}f_A^* \rVert_{\overline{A}} + \lVert D_{\frac{1}{2}}f_B^* \rVert_{\overline{B}} \lesssim \lVert f_A^* \rVert_{\overline{A}} + \lVert f_B^* \rVert_{\overline{B}} = \lVert f_A \rVert_A + \lVert f_B \rVert_B.
	\end{equation*}
	By taking the infimum over all decompositions of $f$, we obtain
	\begin{equation*}
		\lVert f^* \rVert_{\overline{A}+\overline{B}} \lesssim \lVert f \rVert_{A+B}.
	\end{equation*}
	
	As for the properties of the functional $f \mapsto \lVert f^* \rVert_{\overline{A} + \overline{B}}$, \ref{r.i.}, \ref{P2}, and \ref{P4} are rather clear. The same is true for the parts \ref{Q1a} and \ref{Q1b} of \ref{Q1}. As for \ref{Q1c}, we have already recalled that the dilation operator $D_{\frac{1}{2}}$ is bounded on the representation spaces $\overline{A}, \overline{B}$ (see \cite[Theorem~3.23]{NekvindaPesa24}) and therefore also on their sum. Hence, recalling Lemma~\ref{LemmIntSumBasic} and \cite[Chapter~2, Proposition~1.7]{BennettSharpley88}, we may perform a similar computation as above (for $f, g \in \mathcal{M}(\mathcal{R}, \mu)$):
	\begin{equation*}
		\lVert (f+g)^* \rVert_{\overline{A}+\overline{B}} \leq \lVert D_{\frac{1}{2}}(f^* + g^*) \rVert_{\overline{A}+\overline{B}} \lesssim \lVert f^* \rVert_{\overline{A}+\overline{B}} +  \lVert g^* \rVert_{\overline{A}+\overline{B}}.
	\end{equation*}
	As for \ref{P5}, the equivalence of the statements follows quite easily from the equivalence of the functionals in \eqref{ThmSumRIqBFS:E0}.
\end{proof}

When the underlying measure space is completely atomic and all atoms have the same measure (following the convention of \cite[Proof of Theorem~3.1]{MusilovaNekvinda24} we shall denote this measure $\beta \in (0, \infty)$), the situation is more problematic. The problem is that a function on $([0, \infty), \lambda)$ is the non-increasing rearrangement of some sequence on $(\mathcal{R}, \mu)$ if and only if it  is constant on the intervals $[n\beta, (n+1)\beta)$; $n \in \mathbb{N}$, $(n+1)\beta \leq \mu(\mathcal{R} )$ (and zero on $[\mu(\mathcal{R} ), \infty)$ whenever $\mu(\mathcal{R} ) < \infty$, which is easily solvable and thus not the important part at the moment). If a function does hot have this property, then what one can do is to take the integral average over said intervals, as in e.g.~\cite[Proof of Proposition~3.4]{PesaRepreACqN}; however, this approach is compatible with the construction of the canonical representation quasinorm if and only if the function is a~priori known to be non-increasing. As it happens, the functions obtained from the definition of $\overline{A} + \overline{B}$ in general do not have either of those properties. The only known method of dealing with this problem is to use the Hardy--Littlewood--Pólya relation (as in e.g.~\cite[Chapter~2, Proof of Theorem~4.10]{BennettSharpley88}), which in turn requires us to assume that the spaces $A$ and $B$ satisfy the Hardy--Littlewood--Pólya principle (see e.g.~\cite[Definitions~2.15 and 2.16]{PesaRepreACqN} which are relevant to our setting, or e.g.~\cite[Definition~3.5]{BennettSharpley88} for the classical setting). However, this is a very strong assumption, which is well illustrated by the fact that it implies directly that $\lVert \cdot \rVert_{A+B}$ is a rearrangement-invariant quasinorm (via an argument virtually identical to that used in the Proof of Theorem~\ref{ThmSumIntRIBFS}). Hence, our only contribution is the representation-type formula, but given the strength of the assumptions, it is much less interesting than Theorem~\ref{ThmSumRIqBFS}.

\begin{theorem} \label{ThmSumRIqBFS_atom}
	Let $A, B$ be r.i.~quasi-Banach function spaces over a resonant and completely atomic measure space $(\mathcal{R}, \mu)$. Assume that the Hardy--Littlewood--Pólya principle holds for both $A$ and $B$. Denote by $\lVert \cdot \rVert_{\overline{A}}$ and $\lVert \cdot \rVert_{\overline{B}}$ the corresponding canonical representation quasinorms. Then we have for every $f \in \mathcal{M}(\mathcal{R}, \mu)$ that
	\begin{equation} \label{ThmSumRIqBFS_atom:E0}
		\lVert f \rVert_{A+B} \approx \lVert f^* \rVert_{\overline{A} + \overline{B}}.
	\end{equation}
	Furthermore, the functional $f \mapsto \lVert f^* \rVert_{\overline{A} + \overline{B}}$ always has the properties \ref{Q1}, \ref{P2}, \ref{P4}, and \ref{r.i.}, while the validity of \ref{P5} is equivalent to the same property for $A+B$.
\end{theorem}

\begin{proof}
	We only need to show \eqref{ThmSumRIqBFS_atom:E0} as the properties of the functional $f \mapsto \lVert f^* \rVert_{\overline{A} + \overline{B}}$ are proved exactly the same way as in Theorem~\ref{ThmSumRIqBFS}. The general approach is similar to that in the said theorem, so we will be somewhat briefer.
	 
	First, we recall that it follows from \cite[Proposition~3.3]{PesaRepreACqN} that the canonical representation quasinorms $\lVert \cdot \rVert_{\overline{A}}$ and $\lVert \cdot \rVert_{\overline{B}}$ also satisfy the Hardy--Littlewood--Pólya principle and that Lemma~\ref{LemmIntSumBasic} implies that we only have to consider non-negative functions $f \in \mathcal{M}(\mathcal{R}, \mu)$.
	
	Let now
	\begin{align*}
		f^* &\in \overline{A} + \overline{B}, \\
		\alpha &= \lim_{t \to \infty} f^*(t) \in [0, \infty), \\
		E_{f,\alpha} &= \left\{ x \in \mathcal{R}; \; f(x) > \alpha \right\},
	\end{align*}
	where we consider an arbitrary representative of $f$,
	\begin{align*}
		f_{\alpha} &= \max\{ f - \alpha, \, 0 \}, \\
		f_{\infty} &= \min \{f, \, \alpha\}
	\end{align*}	
	and observe that $f = f_{\alpha} + f_{\infty}$, $f_{\infty} \leq \alpha$ $\mu$-a.e., $f_{\alpha} = f_{\alpha} \chi_{E_{f,\alpha}}$ $\mu$-a.e., and
	\begin{align*}
		f_{\alpha}^* &= \max\{f^* - \alpha, \, 0\} = f^* - \alpha, &\lim_{t \to \infty} f_{\alpha}^*(t) = 0.
	\end{align*}
	Thence, it is easy to construct an ordering of atoms in $E_{f, \alpha}$, denoted $e_n$, $n \in \mathbb{N}$, $n \beta < \mu(E_{f, \alpha})$ (note that in our notation $0 \in \mathbb{N}$ and $\beta \in (0,\infty)$ is the measure of each atom in $\mathcal{R}$, as above), such that $f_{\alpha}(e_n) = f_{\alpha}^*(n\beta)$. We then consider some functions $(f_{\alpha}^*)_{\overline{A}}, (f_{\alpha}^*)_{\overline{B}}$ that are as in \eqref{ThmSumRIqBFS:E1} but otherwise arbitrary. We assume (without loss of generality) that $(f_{\alpha}^*)_{\overline{A}}$ and $(f_{\alpha}^*)_{\overline{B}}$ are both zero $\lambda$-a.e.~outside the support of $f_{\alpha}^*$. Next, we consider the functions
	\begin{align*}
		\widetilde{(f_{\alpha}^*)_{\overline{A}}} &= \sum_{n=0}^{\infty} \chi_{[n \beta, (n+1) \beta)} \frac{1}{\beta} \int_{n \beta}^{(n+1)\beta} (f_{\alpha}^*)_{\overline{A}} \: d\lambda, \\
		\widetilde{(f_{\alpha}^*)_{\overline{B}}} &= \sum_{n=0}^{\infty} \chi_{[n \beta, (n+1) \beta)} \frac{1}{\beta} \int_{n \beta}^{(n+1)\beta} (f_{\alpha}^*)_{\overline{B}} \: d\lambda.
	\end{align*}
	Since $\lVert \cdot \rVert_{\overline{A}}$ and $\lVert \cdot \rVert_{\overline{B}}$ both satisfy the Hardy--Littlewood--Pólya principle, \cite[Chapter~1, Proposition~3.7]{BennettSharpley88} implies that
	\begin{align*}
		\left \lVert \widetilde{(f_{\alpha}^*)_{\overline{A}}} \right \rVert_{\overline{A}} &\leq \left \lVert (f_{\alpha}^*)_{\overline{A}} \right \rVert_{\overline{A}}, \\
		\left \lVert \widetilde{(f_{\alpha}^*)_{\overline{B}}} \right \rVert_{\overline{B}} &\leq \left \lVert (f_{\alpha}^*)_{\overline{B}} \right \rVert_{\overline{B}}.
	\end{align*}
	It follows that the functions
	\begin{align*}
		f_{\alpha,A}(x) &= \begin{cases}
			\widetilde{(f_{\alpha}^*)_{\overline{A}}}(n\beta) &\text{when } x = e_n \in E_{f,\alpha}, \\
			0 &\text{otherwise;}
		\end{cases} \\
		f_{\alpha,B}(x) &= \begin{cases}
			(f_{\alpha}^*)_{\overline{B}}(n\beta) &\text{when } x = e_n \in E_{f,\alpha}, \\
			0 &\text{otherwise.}
		\end{cases}
	\end{align*}
	satisfy $f_{\alpha,A} \in A$, $f_{\alpha,B} \in B$, $f_{\alpha} = f_{\alpha,A} + f_{\alpha,B}$, and thus
	\begin{equation*} 
		\lVert f_{\alpha} \rVert_{A+B} \leq \lVert f_{\alpha,A} \rVert_A + \lVert f_{\alpha,B} \rVert_B = \lVert (f_{\alpha,A})^* \rVert_{\overline{A}} + \lVert (f_{\alpha,B})^* \rVert_{\overline{B}} \leq \lVert (f_{\alpha}^*)_{\overline{A}} \rVert_{\overline{A}} + \lVert (f_{\alpha}^*)_{\overline{B}} \rVert_{\overline{B}}.
	\end{equation*}
	The rest of the proof of the estimate
	\begin{equation*}
		\lVert f \rVert_{A+B} \lesssim \lVert f^* \rVert_{\overline{A}+\overline{B}}
	\end{equation*}
	is then identical as in Theorem~\ref{ThmSumRIqBFS}, which is also true for the proof of the converse estimate.
\end{proof}

It remains to show that the representation formula holds with equality in the cases when the space in question is a~priori known to be rearrangement-invariant. The case of the intersection of r.i.~quasi-Banach function spaces is easier, as usual. In fact, the representation itself is rather trivial, the only somewhat interesting observation is that $\overline{A} \cap \overline{B}$ coincides with the canonical representation space of $A \cap B$.

\begin{proposition} \label{PropIntRIqBFS}
	Let $A, B$ be r.i.~quasi-Banach function spaces over a resonant measure space $(\mathcal{R}, \mu)$. Denote by $\lVert \cdot \rVert_{\overline{A}}$ and $\lVert \cdot \rVert_{\overline{B}}$ the corresponding canonical representation quasinorms. Then we have for every $f \in \mathcal{M}(\mathcal{R}, \mu)$ that
	\begin{equation} \label{PropIntRIqBFS:E1}
		\lVert f \rVert_{A \cap B} = \lVert f^* \rVert_{\overline{A} \cap \overline{B}}.
	\end{equation}
	Furthermore, $\lVert \cdot \rVert_{\overline{A} \cap \overline{B}}$ coincides with the canonical representation quasinorm of $\lVert \cdot \rVert_{A \cap B}$.
\end{proposition}

\begin{proof}
	\eqref{PropIntRIqBFS:E1} follows by a direct calculation. Adopting the notation of \cite[Definition~3.1]{PesaRepreACqN}, we observe that it holds for every $g \in \mathcal{M}([0, \infty), \lambda)$ that
	\begin{equation*}
		\begin{split}
			\lVert g \rVert_{\overline{A \cap B}} &= \lVert g^* \chi_{[0, \mu(\mathcal{R}))} \rVert_{\overline{(A \cap B)}_0} \\
			&= \left \lVert \widetilde{g} \right \rVert_{A \cap B}  \\
			&= \max \left\{ \left\lVert \widetilde{g} \right\rVert_{A}, \, \left\lVert \widetilde{g} \right\rVert_{B} \right\} \\
			&= \max \{ \lVert g^* \chi_{[0, \mu(\mathcal{R}))} \rVert_{\overline{A}_0}, \, \lVert g^* \chi_{[0, \mu(\mathcal{R}))} \rVert_{\overline{B}_0} \} \\
			&= \max \left\{ \left\lVert g \right\rVert_{\overline{A}}, \, \left\lVert g \right\rVert_{\overline{B}} \right\} \\
			&= \lVert g \rVert_{\overline{A} \cap \overline{B}}.
		\end{split}
	\end{equation*}
	Here, the function $\widetilde{g} \in \mathcal{M}(\mathcal{R}, \mu)$ is defined as
	\begin{equation*}
		\widetilde{g} = T (g^* \chi_{[0, \mu(\mathcal{R}))}),
	\end{equation*}
	where the operator $T$ is defined as in \cite[Definition~3.1]{PesaRepreACqN}. We note, that since the quasinorms $\lVert \cdot \rVert_{\overline{A}}$, $\lVert \cdot \rVert_{\overline{B}}$, and $\lVert \cdot \rVert_{\overline{A} \cap \overline{B}}$ do not depend on the particular choice of the measure-preserving mapping $\sigma$ or enumeration of atoms used in the definition of $T$, as appropriate, we may consider it to be fixed. 
\end{proof}

As for the sum of r.i.~Banach function spaces, the crucial difference from the situation in Theorems~\ref{ThmSumRIqBFS} and \ref{ThmSumRIqBFS_atom} is that
\begin{enumerate}
	\item the Hardy--Littlewood--Pólya principle holds for all r.i.~Banach function spaces (see e.g.~\cite[Chapter~2, Theorem~4.6]{BennettSharpley88}) so there is no need for different assumptions in the non-atomic and completely atomic cases,
	\item we may work with the Luxemburg construction of the representation norm (as originally presented in \cite{Luxemburg67}; see also \cite[Chapter~2, Theorem~4.10]{BennettSharpley88} for a modern presentation), which allows us to prove the representation formula with equality.
\end{enumerate}

\begin{theorem} \label{ThmSumRIBFS}
	Let $A, B$ be r.i.~Banach function spaces over a resonant measure space $(\mathcal{R}, \mu)$. Denote by $\lVert \cdot \rVert_{\overline{A}}$ and $\lVert \cdot \rVert_{\overline{B}}$ the corresponding canonical representation norms. Then we have for every $f \in \mathcal{M}(\mathcal{R}, \mu)$ that
	\begin{equation*}
		\lVert f \rVert_{A+B} = \lVert f^* \rVert_{\overline{A} + \overline{B}}.
	\end{equation*}
\end{theorem}

\begin{proof}
	Since both $A+B$ and $\overline{A} + \overline{B}$ are r.i.~Banach function spaces (see Theorem~\ref{ThmSumIntRIBFS}), we know from \cite[Section~3]{PesaRepreACqN} that their canonical representation norms are the ones constructed via the Luxemburg construction (as originally presented in \cite{Luxemburg67}; see also \cite[Chapter~2, Theorem~4.10]{BennettSharpley88} for a modern presentation), i.e.
	\begin{align*}
		\lVert f \rVert_{A+B} &= \sup_{\substack{g \in \mathcal{M}(\mathcal{R}, \mu) \\ \lVert g \rVert_{(A+B)'} \leq 1}} \int_0^{\infty} f^* g^* \: d\lambda &\text{for every } f \in  \mathcal{M}(\mathcal{R}, \mu), \\
		\lVert \psi \rVert_{\overline{A} + \overline{B}} &= \sup_{\substack{\varphi \in \mathcal{M}([0, \infty), \lambda) \\ \lVert \varphi \rVert_{\left( \overline{A} + \overline{B} \right)' } \leq 1}} \int_0^{\infty} \psi^* \varphi^* \: d\lambda &\text{for every } \psi \in \mathcal{M}([0, \infty), \lambda),
	\end{align*}
	where $X'$ is the associate space of $X$ (see the comment after Theorem~\ref{ThmSumIntRIBFS}). Using Theorem~\ref{ThmSumIntRIBFS} and \cite[Chapter~2, Theorem~4.10]{BennettSharpley88}, we compute for arbitrary $g \in \mathcal{M}(\mathcal{R}, \mu)$
	\begin{equation*}
		\lVert g \rVert_{(A+B)'} = \max \{\lVert g \rVert_{A'}, \, \lVert g \rVert_{B'}\} = \max \{\lVert g^* \rVert_{\left( \overline{A} \right)'}, \, \lVert g^* \rVert_{\left( \overline{B} \right)'}\} = \lVert g^* \rVert_{\left( \overline{A} \right)' \cap \left( \overline{B} \right)' } = \lVert g^* \rVert_{\left( \overline{A} + \overline{B} \right)' }.
	\end{equation*}
	Whence
	\begin{equation} \label{ThmSumRIBFS:E1}
		\lVert f \rVert_{A+B} = \sup_{\substack{g \in \mathcal{M}(\mathcal{R}, \mu) \\ \lVert g^* \rVert_{\left( \overline{A} + \overline{B} \right)' } \leq 1}} \int_0^{\infty} f^* g^* \: d\lambda \leq \sup_{\substack{\varphi \in \mathcal{M}([0, \infty), \lambda) \\ \lVert \varphi \rVert_{\left( \overline{A} + \overline{B} \right)' } \leq 1}} \int_0^{\infty} f^* \varphi^* \: d\lambda.
	\end{equation}
	
	Now, it is of course not true in general that every non-increasing right-continuos function $\varphi \in \mathcal{M}([0, \infty), \lambda)$ is the non-increasing rearrangement of some function in $\mathcal{M}(\mathcal{R}, \mu)$, so we do not immediately get the inverse inequality. However, we may proceed as in \cite[Chapter~2, Proof of Theorem~4.10]{BennettSharpley88}: Since $\left( \overline{A} + \overline{B} \right)'$ is an r.i.~Banach function space (see e.g.~\cite[Chapter~1, Theorem~2.2]{BennettSharpley88} together with Theorem~\ref{ThmSumIntRIBFS}, or \cite[Theorem~3.1]{GogatishviliSoudsky14}), we may assume that $\varphi$ is non-increasing. By choosing a proper representative, we may assume that it is right-continuous. We may also assume that it is supported on $[0, \mu(\mathcal{R}))$ (because $f^*$ is). If $(\mathcal{R}, \mu)$ is non-atomic, then there is a measure preserving map $\sigma  \colon \mathcal{R} \to [0, \mu(\mathcal{R}))$ (see e.g.~\cite[Lemma~3.2]{MusilovaNekvinda24}) and so we have $\varphi = (\varphi \circ \sigma)^*$. If $(\mathcal{R}, \mu)$ is completely atomic, we denote the measure of each atom $\beta \in (0, \infty)$ and consider the function $\widetilde{\varphi}$ given by
	\begin{equation*}
		\widetilde{\varphi} = \sum_{n=0}^{\infty} \chi_{[n \beta, (n+1) \beta)} \frac{1}{\beta} \int_{n \beta}^{(n+1) \beta} \varphi \: d\lambda.
	\end{equation*}
	Then
	\begin{equation*}
		\int_0^{\infty} f^* \varphi \: d\lambda = \int_0^{\infty} f^* \widetilde{\varphi} \: d\lambda,
	\end{equation*}
	because $f^*$ is constant on the intervals $[n \beta, (n+1) \beta)$, and $\lVert \widetilde{\varphi} \rVert_{\left( \overline{A} + \overline{B} \right)' } \leq \lVert \varphi \rVert_{\left( \overline{A} + \overline{B} \right)' }$, as follows from \cite[Chapter~2, Theorem~4.8]{BennettSharpley88} and the fact that $\left( \overline{A} + \overline{B} \right)'$ is an r.i.~Banach function space (see above; as in Theorem~\ref{ThmSumRIqBFS_atom} this step is an application to the Hardy--Littlewood--Pólya principle).
	
	Hence, the suprema in \eqref{ThmSumRIBFS:E1} indeed coincide and we get
	\begin{equation*}
		\lVert f \rVert_{A+B} = \sup_{\substack{g \in \mathcal{M}(\mathcal{R}, \mu) \\ \lVert g^* \rVert_{\left( \overline{A} + \overline{B} \right)' } \leq 1}} \int_0^{\infty} f^* g^* \: d\lambda = \sup_{\substack{\varphi \in \mathcal{M}([0, \infty), \lambda) \\ \lVert \varphi \rVert_{\left( \overline{A} + \overline{B} \right)' } \leq 1}} \int_0^{\infty} f^* \varphi^* \: d\lambda = \lVert f^* \rVert_{\overline{A} + \overline{B}}.
	\end{equation*}
\end{proof}

\section{The relationship to Wiener--Luxemburg amalgams} \label{SecAmalgams}

In this section we assume that the underlying measure space $(\mathcal{R}, \mu)$ is non-atomic and of infinite measure. As per \cite[Section~3]{MusilovaNekvinda24}, this assumption implies that the representation spaces defined over $([0,\infty), \lambda)$ are uniquely determined. We may thus use the representation theory developed in said paper to extend the original definition from \cite{Pesa22} (which assumes that the underlying measure space is $([0,\infty), \lambda)$) without loss of information. Those facts make non-atomic measure spaces of infinite measure the natural setting for Wiener--Luxemburg amalgams.

\begin{definition}\label{DefAmal}
	Let $(\mathcal{R}, \mu)$ be non-atomic measure space of infinite measure. Let $\lVert \cdot \rVert_A$ and $\lVert \cdot \rVert_B$ be r.i.~quasi-Banach function norms over $(\mathcal{R}, \mu)$ and denote by $\lVert \cdot \rVert_{\overline{A}}$ and $\lVert \cdot \rVert_{\overline{B}}$ the corresponding canonical representation quasinorms. The Wiener--Luxemburg amalgam space $WL(A,B)$ is defined as the space
	\begin{equation*}
		WL(A,B) = \left\{f \in \mathcal{M}(\mathcal{R}, \mu); \; \lVert f^* \chi_{[0,1]} \rVert_{\overline{A}} + \lVert f^* \chi_{(1, \infty)} \rVert_{\overline{B}} \right\},
	\end{equation*}
	equipped with the quasinorm $\lVert \cdot \rVert_{WL(A,B)}$ which is defined for every $f \in \mathcal{M}(\mathcal{R}, \mu)$ by
	\begin{equation}
		\lVert f \rVert_{WL(A,B)} = \lVert f^* \chi_{[0,1]} \rVert_{\overline{A}} + \lVert f^* \chi_{(1, \infty)} \rVert_{\overline{B}}.
	\end{equation}
\end{definition}

\begin{remark}
	It follows from \cite[Theorem~3.1]{MusilovaNekvinda24} and \cite[Theorem~5.3]{Pesa22} that $WL\left( \overline{A}, \overline{B} \right)$ (i.e.~the amalgam of the canonical representation spaces containing functions from $\mathcal{M}([0,\infty), \lambda)$) is an r.i.~quasi-Banach function space, whence \cite[Proposition~3.3]{MusilovaNekvinda24} yields that the same is true for $WL(A,B)$ as defined above. When $A$ and $B$ are r.i.~Banach function spaces then one may obtain in a similar fashion that $WL(A,B)$ is also an r.i.~Banach function space, in the sense that there exists an r.i.~Banach function norm equivalent to $\lVert \cdot \rVert_{WL(A,B)}$. To be more specific, one obtains this result by combining \cite[Chapter~2, Theorem~4.10]{BennettSharpley88}, \cite[Corollary~3.6]{Pesa22}, and \cite[Chapter~2, Theorem~4.9]{BennettSharpley88}.
\end{remark}

\begin{remark} \label{RemAmalCanRepre}
	It is also worth noting that if follows immediately from Definition~\ref{DefAmal} that $WL\left( \overline{A}, \overline{B} \right)$ is a representation space for $WL(A,B)$. Since the representation space is in our case unique, this means that $WL\left( \overline{A}, \overline{B} \right)$ is also the canonical representation space.
\end{remark}

\begin{theorem} \label{TheoIntSumFixed}
	Let $\lVert \cdot \rVert_A$ and $\lVert \cdot \rVert_B$ be r.i.~quasi-Banach function norms over a non-atomic measure space of infinite measure. Then
	\begin{equation}
		A \cap B \hookrightarrow WL(A,B) \hookrightarrow A + B. \label{TheoIntSumFixed:E1}
	\end{equation}
	
	Moreover, if we also assume that the local component of $\lVert \cdot \rVert_A$ is stronger than that of $\lVert \cdot \rVert_B$ while the global component of $\lVert \cdot \rVert_B$ is stronger than that of $\lVert \cdot \rVert_A$ (in the sense of \cite[Theorem~5.6]{Pesa22}), we get
	\begin{align}
		A \cap B &= WL(A,B),  \label{TheoIntSumFixed:E2} \\
		A + B &= WL(B,A), \label{TheoIntSumFixed:E3}
	\end{align}
	up to equivalence of quasinorms.
\end{theorem}

A significantly weaker version of Theorem~\ref{TheoIntSumFixed} has been obtained in \cite[Theorem~3.11 and Corollary~3.12]{Pesa22}. Our contribution is that \eqref{TheoIntSumFixed:E3} is proved with equivalent quasinorms instead of just as sets, while at the same time we work in the much wider setting of r.i.~quasi-Banach function spaces over non-atomic measure spaces of infinite measure, while the setting of the original result were r.i.~Banach function spaces over $([0,\infty), \lambda)$.

Let us also note that in the original setting of r.i.~Banach function spaces over $([0,\infty), \lambda)$, the above-mentioned equivalence of quasinorms in \eqref{TheoIntSumFixed:E3} can be obtained from \cite[Corollary~3.12]{Pesa22} by combining \cite[Corollary~3.10]{NekvindaPesa24} with the fact that the sum is in this case a Banach function space (which follows from \cite{Lozanovskii78} or \cite[Lemma~1.12]{CwikelNilsson03}, see the proof of Theorem~\ref{ThmSumIntRIBFS} for details). An alternative approach using only the tools already developed in \cite{Pesa22} has been presented in \cite[Note~(II'), Proof of Corollary~3.2]{Pesa24Diss}. However, neither of these approaches works in the more general setting of r.i.~quasi-Banach function spaces, as both of them depends in some way on the equality $X = X''$, up to equivalence of quasinorms, which holds if and only if $X$ is a Banach function space (see \cite[Theorem~3.1]{GogatishviliSoudsky14}). Thence, a different approach was required, which we now present.

\begin{proof}[Proof of Theorem~\ref{TheoIntSumFixed}]
	We first assume that the underlying measure space is the interval $[0, \infty)$ equipped with the standard Lebesgue measure (denoted $\lambda$) and start with \eqref{TheoIntSumFixed:E1} (the approach is quite similar to that in \cite[Proof of Theorem~3.11]{Pesa22}):
	
	Fix some $f \in \mathcal{M}([0,\infty), \lambda)$. Then
	\begin{equation*}
		\lVert f \rVert_{WL(A, B)} = \lVert f^* \chi_{[0,1]} \rVert_A + \lVert f^* \chi_{(1, \infty)} \rVert_B \leq \lVert f \rVert_A + \lVert f \rVert_B \leq 2 \lVert f \rVert_{A \cap B}
	\end{equation*}
	which establishes the first embedding. 
	
	As for the second embedding, we may assume without loss of generality that $f$ is non-negative (thanks to Lemma~\ref{LemmIntSumBasic}). Consider now functions $f_0$ and $f_{\infty}$ defined by
	\begin{align*}
		f_0 &= \max \{ f - f^*(1), \, 0\}, \\
		f_{\infty} &= \min \{ f, \, f^*(1)\}.
	\end{align*}
	Then $f = f_0 + f_{\infty}$ and thus
	\begin{equation*}
		\begin{split}
			\lVert f \rVert_{A+B} &\leq \lVert f_0 \rVert_A + \lVert f_{\infty} \rVert_B = \lVert f_0^* \rVert_A + \lVert f_{\infty}^* \rVert_B.
		\end{split} 
	\end{equation*}
	Here, we use first the definition of $\lVert \cdot \rVert_{A+B}$ (if $f_0 \notin A$ or $f_{\infty} \notin B$ then the estimate holds trivially) and then invariance of both $\lVert \cdot \rVert_A$ and $\lVert \cdot \rVert_B$. Furthermore, thanks to $f$ being non-negative, it is an exercise to verify that
	\begin{align*}
		f_0^* &= \max\{f^* - f^*(1), \, 0\} = (f^* - f^*(1)) \chi_{[0,1]}, \\
		f_{\infty}^* &= \min\{f^*, \, f^*(1)\} = f^*(1) \chi_{[0,1]} + f^* \chi_{(1, \infty)},
	\end{align*}
	and therefore
	\begin{equation*}
		\begin{split}
			\lVert f \rVert_{A+B} &\lesssim \lVert f^* \chi_{[0,1]} \rVert_A + \lVert f^*(1) \chi_{[0,1]} \rVert_B + \lVert f^* \chi_{(1, \infty)} \rVert_B \\
			&\leq \lVert f \rVert_{WL(A, B)} + \lVert \chi_{[0,1]} \rVert_B \lVert \chi_{[0,1]} \rVert_A^{-1} \lVert f^*(1) \chi_{[0,1]} \rVert_A \\
			&\leq (1 + \lVert \chi_{[0,1]} \rVert_B \lVert \chi_{[0,1]} \rVert_A^{-1}) \lVert f \rVert_{WL(A, B)}.
		\end{split} 
	\end{equation*}
	This establishes the second embedding.
	
	As for the validity of \eqref{TheoIntSumFixed:E2} under the appropriate assumptions, \cite[Theorem~5.6 and Remark~5.2]{Pesa22} together with said assumptions imply	
	\begin{align*}
		WL(A,B) &\hookrightarrow A, \\
		WL(A,B) &\hookrightarrow B,
	\end{align*}
	which is clearly sufficient for the remaining embedding.
	
		It remains to show that \eqref{TheoIntSumFixed:E3} holds with equivalent quasinorms (under the appropriate assumptions). Considering \eqref{TheoIntSumFixed:E1}, we only have to show that $A + B \hookrightarrow WL(B,A)$. Similarly to \cite[Proof of Corollary~3.12]{Pesa22}, it follows from our assumptions via \cite[Theorem~5.6 and Remark~5.2]{Pesa22} that
	\begin{equation} \label{TheoIntSumFixed:E4}
		\begin{split}
			A = WL(A,A) &\hookrightarrow WL(B, A),  \\
			B = WL(B,B) &\hookrightarrow WL(B, A), 
		\end{split}
	\end{equation}
	where the equalities hold up to equivalence of quasinorms. Consider thus arbitrary $f \in A + B$ and any pair of functions $f_A \in A$, $f_B \in B$ such that $f = f_A + f_B$. Then
	\begin{equation*}
		\lVert f \rVert_{WL(B,A)} \lesssim  \lVert f_A \rVert_{WL(B,A)} + \lVert f_B \rVert_{WL(B,A)} \lesssim \lVert f_A \rVert_A + \lVert f_B \rVert_B.
	\end{equation*}
	Since the constants in this estimate depend only on the modulus of concavity of $\lVert \cdot \rVert_{WL(A,B)}$ and the constants from the embeddings in \eqref{TheoIntSumFixed:E4}, we may take the infimum over all such decompositions to obtain the desired estimate
	\begin{equation*}
		\lVert f \rVert_{WL(A,B)} \lesssim 	\lVert f \rVert_{A+B}.
	\end{equation*}
	
	Finally, we consider the case when the underlying measure space is non-atomic and of infinite measure. We observe that $\overline{WL(A,B)} = WL\left( \overline{A}, \overline{B} \right)$ (Remark~\ref{RemAmalCanRepre}) and $\overline{A \cap B} = \overline{A} \cap \overline{B}$ (Proposition~\ref{PropIntRIqBFS}), in both cases with equal quasinorms. We may therefore apply the special case proven above to compute for arbitrary $f \in \mathcal{M}(\mathcal{R, \mu})$ that
	\begin{equation*}
		\lVert f \rVert_{A \cap B} = \lVert f^* \rVert_{\overline{A} \cap \overline{B}} \approx \lVert f^* \rVert_{WL \left( \overline{A}, \overline{B} \right)} = \lVert f \rVert_{WL(A,B)}.
	\end{equation*}
	As for the sum, we again observe that $\overline{WL(B,A)} = WL(\overline{B}, \overline{A})$ and so we may employ Theorem~\ref{ThmSumRIqBFS} to compute for arbitrary $f \in \mathcal{M}(\mathcal{R, \mu})$ that
	\begin{equation*}
		\lVert f \rVert_{A + B} \approx \lVert f^* \rVert_{\overline{A} + \overline{B}} \approx \lVert f^* \rVert_{WL \left( \overline{B}, \overline{A} \right)} = \lVert f \rVert_{WL(B,A)}.
	\end{equation*}
\end{proof}

Since $WL(A,B)$ is always an r.i.~quasi-Banach function space, we immediately obtain the promised partial solution to the question whether $A+B$ has the Fatou property \ref{P3}. 

\begin{corollary} \label{CorSum_qBFS}
	Let $\lVert \cdot \rVert_A$ and $\lVert \cdot \rVert_B$ be r.i.~quasi-Banach function norms (over a non-atomic measure space of infinite measure) such that the local component of $\lVert \cdot \rVert_A$ is stronger than that of $\lVert \cdot \rVert_B$ while the global component of $\lVert \cdot \rVert_B$ is stronger than that of $\lVert \cdot \rVert_A$ (in the sense of \cite[Theorem~5.6]{Pesa22}). Then $A+B$ is an r.i.~quasi-Banach function space, meaning that there exists an r.i.~quasi-Banach function norm that is equivalent to $\lVert \cdot \rVert_{A+B}$.
\end{corollary}



\bibliographystyle{dabbrv}
\bibliography{bibliography}

@book {Pesa24Diss,
    AUTHOR = {Pe\v{s}a, Dalimil},
     TITLE = {Fine properties of certain specific function spaces},
 PUBLISHER = {Thesis, Karlova Univerzita},
      YEAR = {2024},
     PAGES = {218},
}

@article{Pesa22,
    AUTHOR = {Pe\v{s}a, Dalimil},
     TITLE = {Wiener--{L}uxemburg amalgam spaces},
   JOURNAL = {J. Funct. Anal.},
  FJOURNAL = {Journal of Functional Analysis},
    VOLUME = {282},
      YEAR = {2022},
    NUMBER = {1},
     PAGES = {Paper No. 109270, 47},
      ISSN = {0022-1236},
   MRCLASS = {46E30 (46A16)},
  MRNUMBER = {4323512},
MRREVIEWER = {Khedoudj Saibi},
       DOI = {10.1016/j.jfa.2021.109270},
}

@Article{NekvindaPesa24,
    AUTHOR = {Nekvinda, Ale{\v s} and Pe{\v s}a, Dalimil},
     TITLE = {On the properties of quasi-{B}anach function spaces},
   JOURNAL = {J. Geom. Anal.},
  FJOURNAL = {Journal of Geometric Analysis},
    VOLUME = {34},
      YEAR = {2024},
    NUMBER = {8},
     PAGES = {Paper No. 231, 29},
      ISSN = {1050-6926,1559-002X},
   MRCLASS = {46A16 (46E30)},
  MRNUMBER = {4746146},
       DOI = {10.1007/s12220-024-01673-y},
}

@article{MusilovaNekvinda24,
	AUTHOR = {Musilov\'{a}, Anna and Nekvinda, Ale\v{s} and Pe\v{s}a, Dalimil and
              Tur\v{c}inov\'{a}, Hana},
     TITLE = {On the properties of rearrangement-invariant quasi-{B}anach
              function spaces},
   JOURNAL = {Nonlinear Anal.},
  FJOURNAL = {Nonlinear Analysis. Theory, Methods \& Applications. An
              International Multidisciplinary Journal},
    VOLUME = {260},
      YEAR = {2025},
     PAGES = {Paper No. 113854},
      ISSN = {0362-546X},
   MRCLASS = {46E30 (46A16)},
  MRNUMBER = {4919871},
       DOI = {10.1016/j.na.2025.113854},
}

@article{PesaRepreACqN,
    title={Absolute continuity of the (quasi)norm in rearrangement-invariant spaces}, 
      author={Peša, Dalimil},
      year={2024},
      eprint={2412.13903},
      archivePrefix={arXiv},
      primaryClass={math.FA},
      url={https://arxiv.org/abs/2412.13903}, 
}

@book{BennettSharpley88,
    AUTHOR = {Bennett, Colin and Sharpley, Robert},
     TITLE = {Interpolation of operators},
    SERIES = {Pure and Applied Mathematics},
    VOLUME = {129},
 PUBLISHER = {Academic Press, Inc., Boston, MA},
      YEAR = {1988},
     PAGES = {xiv+469},
      ISBN = {0-12-088730-4},
   MRCLASS = {46-02 (46E30 46Exx 46M35)},
  MRNUMBER = {928802},
MRREVIEWER = {Mario Milman},
}

@article{GogatishviliSoudsky14,
    AUTHOR = {Gogatishvili, Amiran and Soudsk\'{y}, Filip},
     TITLE = {Normability of {L}orentz spaces---an alternative approach},
   JOURNAL = {Czechoslovak Math. J.},
  FJOURNAL = {Czechoslovak Mathematical Journal},
    VOLUME = {64(139)},
      YEAR = {2014},
    NUMBER = {3},
     PAGES = {581--597},
      ISSN = {0011-4642},
   MRCLASS = {46E30},
  MRNUMBER = {3298548},
MRREVIEWER = {Karol Le\'{s}nik},
       DOI = {10.1007/s10587-014-0120-y},
}

@book {Luxemburg55,
    AUTHOR = {Luxemburg, Wilhelmus Anthonius Josephus},
     TITLE = {Banach function spaces},
 PUBLISHER = {Thesis, Technische Hogeschool te Delft},
      YEAR = {1955},
     PAGES = {70},
   MRCLASS = {46.2X},
  MRNUMBER = {0072440},
MRREVIEWER = {B. R. Gelbaum},
}

@book {KreinPetunin82,
    AUTHOR = {Kre\u{\i}n, S. G. and Petun\={\i}n, Yu. \={I}. and Sem\"{e}nov, E. M.},
     TITLE = {Interpolation of linear operators},
    SERIES = {Translations of Mathematical Monographs},
    VOLUME = {54},
      NOTE = {Translated from the Russian by J. Sz\H{u}cs},
 PUBLISHER = {American Mathematical Society, Providence, R.I.},
      YEAR = {1982},
     PAGES = {xii+375},
      ISBN = {0-8218-4505-7},
   MRCLASS = {46M35 (46Exx)},
  MRNUMBER = {649411},
MRREVIEWER = {E. Gerlach},
}

@book {BerghLofstrom76,
    AUTHOR = {Bergh, J\"{o}ran and L\"{o}fstr\"{o}m, J\"{o}rgen},
     TITLE = {Interpolation spaces. {A}n introduction},
    SERIES = {Grundlehren der Mathematischen Wissenschaften, No. 223},
 PUBLISHER = {Springer-Verlag, Berlin-New York},
      YEAR = {1976},
     PAGES = {x+207},
   MRCLASS = {46M35},
  MRNUMBER = {0482275},
}

@article {CwikelNilsson03,
    AUTHOR = {Cwikel, Michael and Nilsson, Per G. and Schechtman, Gideon},
     TITLE = {Interpolation of weighted {B}anach lattices. {A}
              characterization of relatively decomposable {B}anach lattices},
   JOURNAL = {Mem. Amer. Math. Soc.},
  FJOURNAL = {Memoirs of the American Mathematical Society},
    VOLUME = {165},
      YEAR = {2003},
    NUMBER = {787},
     PAGES = {vi+127},
      ISSN = {0065-9266},
   MRCLASS = {46B70 (46B42 46E30 46M35)},
  MRNUMBER = {1996919},
MRREVIEWER = {Mieczys\l aw Masty\l o},
       DOI = {10.1090/memo/0787},
}

@incollection {Lozanovskii78,
    AUTHOR = {Lozanovski\u{\i}, G. Ya.},
     TITLE = {Transformations of ideal {B}anach spaces by means of concave
              functions},
 BOOKTITLE = {Qualitative and approximate methods for the investigation of
              operator equations, {N}o. 3 ({R}ussian)},
     PAGES = {122--148},
 PUBLISHER = {Yaroslav. Gos. Univ., Yaroslavl'},
      YEAR = {1978},
   MRCLASS = {46E30},
  MRNUMBER = {559326},
MRREVIEWER = {Yu. A. Abramovich},
}

@article{LoristNieraeth23,
    AUTHOR = {Lorist, Emiel and Nieraeth, Zoe},
     TITLE = {Banach function spaces done right},
   JOURNAL = {Indag. Math. (N.S.)},
  FJOURNAL = {Koninklijke Nederlandse Akademie van Wetenschappen.
              Indagationes Mathematicae. New Series},
    VOLUME = {35},
      YEAR = {2024},
    NUMBER = {2},
     PAGES = {247--268},
      ISSN = {0019-3577},
   MRCLASS = {46-02 (46A16 46E30)},
  MRNUMBER = {4726599},
       DOI = {10.1016/j.indag.2023.11.004},
}

@incollection{Luxemburg67,
    AUTHOR = {Luxemburg, Wilhelmus Anthonius Josephus},
     TITLE = {Rearrangement-invariant {B}anach function spaces},
 BOOKTITLE = {Proc. {S}ympos. in Analysis},
    SERIES = {{Q}ueen's {P}apers in {P}ure and {A}ppl. {M}ath.},
	VOLUME = {10},
     PAGES = {83--144},
 PUBLISHER = {Queen's University},
      YEAR = {1967},
}

@article {PeetreSparr72,
    AUTHOR = {Peetre, J. and Sparr, G.},
     TITLE = {Interpolation of normed abelian groups},
   JOURNAL = {Ann. Mat. Pura Appl. (4)},
  FJOURNAL = {Annali di Matematica Pura ed Applicata. Serie Quarta},
    VOLUME = {92},
      YEAR = {1972},
     PAGES = {217--262},
      ISSN = {0003-4622},
   MRCLASS = {46M35 (41A65 43A95 46E35)},
  MRNUMBER = {322529},
MRREVIEWER = {J. Friberg},
       DOI = {10.1007/BF02417949},
}

@article {Calderon66,
    AUTHOR = {Calder\'{o}n, A.-P.},
     TITLE = {Spaces between {$L^{1}$} and {$L^{\infty }$} and the
              theorem of {M}arcinkiewicz},
   JOURNAL = {Studia Math.},
  FJOURNAL = {Polska Akademia Nauk. Instytut Matematyczny. Studia
              Mathematica},
    VOLUME = {26},
      YEAR = {1966},
     PAGES = {273--299},
      ISSN = {0039-3223},
   MRCLASS = {46.35 (46.10)},
  MRNUMBER = {203444},
MRREVIEWER = {W. A. J. Luxemburg},
       DOI = {10.4064/sm-26-3-301-304},
}

@article {Ryff70,
    AUTHOR = {Ryff, John V.},
     TITLE = {Measure preserving transformations and rearrangements},
   JOURNAL = {J. Math. Anal. Appl.},
  FJOURNAL = {Journal of Mathematical Analysis and Applications},
    VOLUME = {31},
      YEAR = {1970},
     PAGES = {449--458},
      ISSN = {0022-247X},
   MRCLASS = {28A65},
  MRNUMBER = {419734},
       DOI = {10.1016/0022-247X(70)90038-7},
       URL = {https://doi-org.ezproxy.is.cuni.cz/10.1016/0022-247X(70)90038-7},
}

@book {Zaanen67,
    AUTHOR = {Zaanen, Adriaan Cornelis},
     TITLE = {Integration},
      NOTE = {Completely revised edition of An introduction to the theory of
              integration},
 PUBLISHER = {North-Holland Publishing Co., Amsterdam; Interscience
              Publishers John Wiley \& Sons, Inc., New York},
      YEAR = {1967},
     PAGES = {xiii+604},
   MRCLASS = {28.00 (46.00)},
  MRNUMBER = {222234},
MRREVIEWER = {N. Dinculeanu},
}
\end{document}